\documentclass[11pt]{amsart}
\usepackage{graphicx,color}

\theoremstyle{plain}

\theoremstyle{plain}
\newtheorem{theorem}{Theorem} [section]

\newtheorem{lemma}[theorem]{Lemma}
\newtheorem{proposition}[theorem]{Proposition}

\theoremstyle{definition}

\theoremstyle{remark}

\numberwithin{theorem}{section}
\numberwithin{equation}{section}
\numberwithin{figure}{section}

\def\mean#1{\mathchoice
         {\mathop{\kern 0.2em\vrule width 0.6em height 0.69678ex depth -0.58065ex
                 \kern -0.8em \intop}\nolimits_{\kern -0.4em#1}}%
         {\mathop{\kern 0.1em\vrule width 0.5em height 0.69678ex depth -0.60387ex
                 \kern -0.6em \intop}\nolimits_{#1}}%
         {\mathop{\kern 0.1em\vrule width 0.5em height 0.69678ex
             depth -0.60387ex
                 \kern -0.6em \intop}\nolimits_{#1}}%
         {\mathop{\kern 0.1em\vrule width 0.5em height 0.69678ex depth -0.60387ex
                 \kern -0.6em \intop}\nolimits_{#1}}}

\def\N{\mathbb N}

\def\R{\mathbb R}

\def\g{\gamma}

\def\e{\varepsilon}

\def\l{\lambda}
\def\L{\Lambda}
\def\O{\Omega}

\title[Stability for the Monge-Amp\`ere equation
and strong convergence of optimal maps]{Second order stability for the Monge-Amp\`ere equation\\
and strong Sobolev convergence of optimal transport maps}

\author[G. De Philippis]{Guido De Philippis}
\address{Scuola Normale Superiore,
p.za dei Cavalieri 7, I-56126 Pisa, Italy}
\email{guido.dephilippis@sns.it}

\author[A. Figalli]{Alessio Figalli}
\address{Department of Mathematics,
The University of Texas at Austin, 1 University Station C1200,
Austin TX 78712, USA}
\email{figalli@math.utexas.edu}


\begin{document}

\begin{abstract}
The aim of this note is to show that Alexandrov solutions
of the Monge-Amp\`ere equation,
with right hand side bounded away from zero and infinity, converge strongly in $W^{2,1}_{\rm loc}$
if their right hand side converge strongly in $L^1_{\rm loc}$.
As a corollary we deduce strong $W^{1,1}_{\rm loc}$ stability of optimal transport maps.
\end{abstract}

\maketitle

\section{Introduction}
Let $\Omega\subset \R^n$ be a bounded convex domain.
Recently \cite{DepFi} the authors showed that convex Alexandrov solutions of
\begin{equation}
\label{MA}
\begin{cases}
\det D^2 u=f \quad &\text{in $\Omega$}\\
u=0 &\text{on $\partial \Omega$},
\end{cases}
\end{equation}
with  $0<\lambda \le f \le \Lambda$, are $W^{2,1}_{\rm loc}(\Omega)$.
More precisely, they were able to prove uniform interior $L\log L$-estimates for $D^2u$.
This result has also been improved in \cite{DFS,S}, where it is actually shown that $u \in W^{2,\g}_{\rm loc}(\Omega)$
for some $\g=\g(n,\l,\L)>1$: more precisely, for any $\Omega'\subset \subset \Omega$,
\begin{equation}\label{ulogu}
\int_{\Omega'} |D^2 u|^{\g}\le C(n,\l,\L,\O,\O').
\end{equation}


A question which naturally arises in view of the previous results is the following:
choose a sequence of functions $f_k$ with $\lambda \le f_k \le \Lambda$ which converges to $f$ strongly in $L^1_{\rm loc}(\Omega)$,
and denoted by $u_k$ and $u$ the solutions of \eqref{MA} corresponding to $f_k$ and $f$ respectively.
By the convexity of $u_k$ and $u$, and the uniqueness of solutions to \eqref{MA},
it is immediate to deduce that $u_k \to u$ uniformly,
and $\nabla u_k \to \nabla u$ in $L^{p}_{\rm loc}(\Omega)$ for any $p<\infty$.
What can be said about the strong convergence of $D^2 u_k$?
Due to the highly nonlinear character of the Monge-Amp\`ere equation, this question is nontrivial.
(Note that weak $W^{2,1}_{\rm loc}$ convergence is immediate by compactness, even under the weaker
assumption that $f_k$ converge to $f$ weakly in $L^1_{\rm loc}(\Omega)$.)
 
The aim of this short note is to prove that actually strong convergence holds. In fact our main result is the following:

\begin{theorem}\label{conv} 
Let $\Omega_k\subset \R^n$ be a  family of convex domains, and  let $u_k:\Omega_k \to \R$ be convex Alexandrov solutions of
\begin{equation}
\label{MA_k}
\begin{cases}
\det D^2 u_k=f_k \quad &\text{in $\Omega_k$}\\
u_k=0 &\text{on $\partial \Omega_k$}
\end{cases}
\end{equation}
with  $0<\lambda \le f_k \le \Lambda$. Assume that $\Omega_k$ converge to some convex domain $\Omega$
in the Hausdorff distance, and $f_k \chi_{\Omega_k}$ converge to $f$ in $L^1_{\rm loc}(\Omega)$. Then, if $u$
denotes the unique Alexandrov solution of 
\[
\begin{cases}
\det D^2 u=f \quad &\text{in $\Omega$}\\
u=0 &\text{on $\partial \Omega$},
\end{cases}
\] 
for any $\Omega'\subset\subset  \Omega$ we have
\begin{equation}\label{L1econv}
\|u_k-u\|_{W^{2,1}(\Omega')}\to 0 \qquad \text{as $k \to \infty$}.
\end{equation}
(Obviously, since the functions $u_k$ are uniformly bounded in $W^{2,\g}(\Omega')$,
this gives strong convergence in $W^{2,\g'}(\Omega')$ for any $\g'<\g$.)
\end{theorem}

As a consequence of the previous theorem we can prove the following stability result for optimal transport maps:
\begin{theorem}\label{transport}
Let $\O_1,\O_2\subset \R^n$ be two bounded domains with $\Omega_2$ convex, and let $f_k,g_k$ be a family of probability densities
such that $0<\l \leq f_k,g_k \le \Lambda$ inside $\O_1$ and $\O_2$ respectively.
Assume that $f_k \to f$ in $L^1(\Omega_1)$ and $g_k \to g$ in $L^1(\Omega_2)$,
and let $T_k:\O_1\to \O_2$ (resp. $T:\O_1\to \O_2$) be the (unique) optimal transport map for the
quadratic cost sending $f_k$ onto $g_k$ (resp. $f$ onto $g$).
Then $T_k\to T$ in  $  W^{1,\g'}_{\rm loc}(\O_1)$ for some $\g'>1$.
\end{theorem}
We point out that, in order to prove \eqref{L1econv} and the local $W^{1,1}$ stability of optimal transport maps,
the interior $L\log L$-estimates from \cite{DepFi} are sufficient. Indeed, the $W^{2,\gamma}$-estimates
are used just to improve the convergence from $W^{2,1}_{\rm loc}$ to $W^{2,\g'}_{\rm loc}$ with $\g'<\g$.

The paper is organized as follows: in the next section we collect some notation and preliminary results.
Then in Section \ref{sect:thm} we prove Theorem \ref{conv}, and in Section \ref{sect:cor} we prove Theorem \ref{transport}.
\\

\textit{Acknowledgments:} We thank Luigi Ambrosio for stimulating our interest on this problem.
We also acknowledge the hospitality at the University of Nice during the workshop ``Geometry meets transport'',
where part of the present work has been done.
AF has been partially supported by NSF Grant DMS-0969962.
Both authors acknowledge the support of the ERC ADG Grant GeMeThNES.

\section{Notation and preliminaries}
\label{sec:notation}

Given a convex function $u:\Omega \to \R$, we define its \textit{Monge-Amp\`ere measure} as 
\[
\mu_u(E):=|\partial u(E)| \qquad \forall\, E \subset \Omega\text{ Borel}
\]
(see \cite[Theorem 1.1.13]{G}), where 
\[
\partial u(E):=\bigcup_{x\in E}\partial u(x).
\]
Here $\partial u(x)$ is the subdifferential of $u$ at $x$,
and $|F|$ denotes the Lebesgue measure of a set $F$.
In case $u\in C_{\rm loc}^{1,1}$, by the Area Formula \cite[Paragraph 3.3]{EG} the following representation holds:
\[
\mu_u= \det D^2 u \,dx.
\]
The main property of Monge-Amp\`ere measure we are going to use is the following (see \cite[Lemma 1.2.2 and Lemma 1.2.3]{G}):
\begin{proposition}\label{uni}
Let $u_k:\Omega \to \R$ be a sequence of convex functions converging locally uniformly to $u$.
Then the associated Monge-Amp\`ere measures $\mu_{u_k}$ converges to $\mu_u$ in duality with
the space of continuous functions compactly supported in $\Omega$. In particular
\[
\mu_u(A)\le \liminf_{k\to \infty}\mu_{u_k}(A)
\]
for any open set $A \subset \Omega$.
\end{proposition}

Given a Radon measure $\nu$ on $\R^n$ and a bounded convex domain $\Omega \subset \R^n$,
we say that a convex function $u:\Omega \to \R$ is an \textit{Alexandrov solution} of the Monge-Amp\`ere equation
$$
\det D^2 u=\nu \quad \text{in }\Omega
$$
if $\mu_u(E)= \nu(E)$ for every Borel set $E\subset \Omega$.\\

If $v:\overline \Omega \to \R$ is a continuous function, 
we define its \textit{convex envelope inside $\Omega$} as 
\begin{equation}
\label{eq:convex envelope}
\Gamma_v(x):=\sup\{ \ell(x)\,:\, \ell \leq v \text{ in }\Omega,\,
\, \ell \text{ affine}\}.
\end{equation}
In case $\Omega$ is a  convex domain and
$v \in C^2(\Omega)$, it is easily seen that 
\begin{equation}
\label{eq:order symm matrices}
 D^2v(x)\ge 0 \qquad \text{for every } x\in \{v=\Gamma_v\}\cap\Omega 
\end{equation}
in the sense of symmetric matrices. Moreover the following inequality between measures holds in $\Omega$:
\begin{equation}\label{misconv}
\mu_{\Gamma_v} \le \det D^2v {\mathbf 1}_{\{v=\Gamma_v\}}\,dx,
\end{equation}
(here ${\mathbf 1}_E$ is the characteristic function of a set $E$).\footnote{To see this, let us first recall that
by \cite[Lemma 6.6.2]{G}, if $x_0\in  \Omega\setminus \{\Gamma_v=v\}$ and $a\in \partial \Gamma_v(x_0)$, then  the convex set 
\[
\{ x\in \Omega:\Gamma_v(x)=a\cdot (x-x_0)+\Gamma_v(x_0)\}
\]
is nonempty and contains more than one point. In particular
\[
\partial \Gamma_v\bigl(\Omega\setminus \{\Gamma_v=v\}\bigr)\subset\{ p\in \R^n \,:\,\text{there exist $x, y \in  \Omega$, $x\ne y$ and
 $p \in \partial \Gamma_v(x)\cap\partial \Gamma_v(y)$}\}.
\]
This last set is  contained in the set of nondifferentiability of the convex conjugate of $\Gamma_v$,
so it has zero Lebesgue measure (see \cite[Lemma 1.1.12]{G}), hence
\begin{equation}
\label{eq:meas zero}
\bigl|\partial \Gamma_v\bigl(\Omega\setminus \{\Gamma_v=v\}\bigr) \bigr| =0.
\end{equation}
Moreover, since $v\in C^1( \Omega)$, for any $x\in \{\Gamma_v=v\}\cap \Omega$ it holds $\partial \Gamma_v (x)=\{\nabla v(x)\}$.
Thus, using \eqref{eq:meas zero}
and \eqref{eq:order symm matrices}, for any open set $A\subset\subset \Omega$ we have
\[
\begin{split}
\mu_{\Gamma_v}(A)&= \bigl|\partial\Gamma_v\bigl(A\cap \{\Gamma_v=v\}\bigr)\bigr|
=\bigl|\nabla v\bigl(A\cap \{\Gamma_v=v\}\bigr)\bigr|\\
&\le \int_{A\cap \{\Gamma_v=v\}}|\det D^2 v|=\int_{A\cap \{\Gamma_v=v\}}\det D^2 v,
\end{split}
\]
as desired.
(The inequality above follows from the Area Formula in \cite[Paragraph 3.3.2]{EG} applied to the $C^1$ map $\nabla v$.)}\\

We recall that a continuous function $v$ is said to be 
\textit{twice differentiable} at $x$ if there exists a (unique) vector $\nabla v(x)$ and a
(unique) symmetric matrix $\nabla^2 v(x)$ such that
\[
v(y)=v(x)+\nabla v(x)\cdot (y-x)+\frac{1}{2}\nabla^2 v(x)[y-x,y-x]+o(|y-x|^2).
\]
In case $v$ is twice differentiable at some point $x_0 \in \{v=\Gamma_v\}$,
then it is immediate to check that
\begin{equation}
\label{eq:order symm matrices2}
\nabla^2 v(x_0) \geq 0.
\end{equation}
By Alexandrov Theorem, any convex function is twice differentiable almost everywhere
(see for instance \cite[Paragraph 6.4]{EG}).
In particular, \eqref{eq:order symm matrices2} holds almost everywhere on $\{v=\Gamma_v\}$
whenever $v$
is the difference of two convex functions.\\

Finally we recall that, in case $v\in W^{2,1}_{\rm loc}$, then the pointwise Hessian of $v$ coincides almost everywhere with its 
distributional Hessian \cite[Sections 6.3 and 6.4]{EG}.
Since in the sequel we are going to deal with $W^{2,1}_{\rm loc}$ convex functions,
we will use $D^2u$ to denote both the pointwise and the distributional Hessian.

\section{Proof of Theorem \ref{conv}}\label{sect:thm}

We are going to use the following result:
\begin{lemma}\label{conenv}
Let $\Omega\subset \R^n$ be a bounded convex domain,
and let  $u,v:\overline\Omega \to \R$ be two continuous strictly convex functions such that
$\mu_u=f\,dx$ and $\mu_v=g\,dx$, with $f,g \in L^1_{\rm loc}(\Omega)$. Then
\begin{equation}\label{gammauv}
\mu_{\Gamma_{u-v}}\le \Big( f^{1/n}-g^{1/n}\Big)^n{\mathbf 1}_{\{u-v=\Gamma_{u-v}\}}\, dx.
\end{equation}
\end{lemma}

\begin{proof}
In case $u,v$ are of class $C^2$ inside $\Omega$, by \eqref{eq:order symm matrices} we have 
\[
0\le D^2 u(x)-D^2v(x)\qquad \text{for every } x\in \{u-v=\Gamma_{u-v}\},
\]
so using the monotonicity and the concavity of the function $\det^{1/n}$ on the cone of non-negative symmetric matrices we get
\[
0 \leq  \det(D^2 u-D^2v)\leq  \left( \big(\det D^2 u\big)^{1/n}-\big(\det D^2v \big)^{1/n} \right)^n\qquad
\text{ on } \{u-v=\Gamma_{u-v}\},
\]  
which combined with \eqref{misconv} gives the desired result.

Now, for the general case, we consider a sequence of smooth uniformly convex domains $\Omega_k$ increasing to $\Omega$,
two sequences of smooth functions $f_k$ and $g_k$ converging respectively to $f$ and $g$ in $L^1_{\rm loc}(\Omega)$,
and we solve 
\[
\begin{cases}
\det D^2 u_k =f_k \quad &\text{in $\Omega_k$}\\
u_k=u\ast \rho_{k} &\text{on $\partial \Omega_k$,}
\end{cases}
\qquad
\begin{cases}
\det D^2 v_k =g_k \quad &\text{in $\Omega_k$}\\
v_k=v\ast \rho_{k} &\text{on $\partial \Omega_k$,}
\end{cases}
\]
where $\rho_k$ is a smooth sequence of convolution kernels. In this way both $u_k$ and $v_k$ are smooth on $\overline\Omega_k$ \cite[Theorem 17.23]{GT},
and $\|u_k-u\|_{L^\infty(\Omega_k)}+\|v_k-v\|_{L^\infty(\Omega_k)} \to
0$ as $k \to \infty$.\footnote{
Indeed, first of all it is easy to see that $u_k$ (resp. $v_k$) converges uniformly to $u$ (resp. $v$)
both on $\partial \Omega_k$ and in any compact subdomain of $\Omega$.
Then, using for instance a contradiction argument,
one exploits the convexity of $u_k$ (resp. $v_k$) and $\Omega_k$,
and 
the uniform continuity of $u$ (resp. $v$),
to show that the convergence is actually uniform on the whole $\Omega_k$.}
Hence, also $\Gamma_{u_k-v_k}$ converges locally uniformly to $\Gamma_{u-v}$. Moreover, it follows easily
from the definition of contact set that
\begin{equation}\label{c}
\limsup_{k \to \infty} {\mathbf 1}_{\{u_k-v_k=\Gamma_{u_k-v_k}\}} \le {\mathbf 1}_{\{u-v=\Gamma_{u-v}\}}.
\end{equation}
We now observe that the previous step applied to $u_k$ and $v_k$ gives
\[
\mu_{\Gamma_{u_k-v_k}} \le
\Bigg( \big(\det D^2 u_k\big)^{1/n}-\big(\det D^2v_k \big)^{1/n} \Bigg)^n {\mathbf 1}_{\{u_k-v_k=\Gamma_{u_k-v_k}\}}\, dx,
\]
Thus, letting $k \to \infty$ and taking in account Proposition \ref{uni} and \eqref{c},
we obtain \eqref{gammauv}.
\end{proof}

\begin{proof}[Proof of Theorem \ref{conv}]
The $L^1_{\rm loc}$ convergence of $u_k $ (resp. $\nabla u_k$) to $u$ (resp. $\nabla u$) is easy and standard,
so we focus on the convergence of the second derivatives.

Without loss of generality we can assume that $\Omega'$ is convex, and that $\Omega'\subset\subset \Omega_k$
(since $\Omega_k \to \Omega$ in the Hausdorff distance, this is always true
for $k$ sufficiently large).
Fix $\e \in (0,1)$, let $\Gamma_{u-(1-\e)u_k}$
be the convex envelope of $u-(1-\e)u_k$ inside $\Omega'$ (see \eqref{eq:convex envelope}),
and define
$$
A_k^\e:=\{x \in \Omega'\,:\,u(x)-(1-\e)u_k(x)=\Gamma_{u-(1-\e)u_k}(x)\}.
$$
Since $u_k \to u$ locally uniformly,
$\Gamma_{u-(1-\e)u_k}$
converges uniformly to $\Gamma_{\e u}=\e u$ (as $u$ is convex) inside $\Omega'$. Hence,
by applying Proposition \ref{uni} and \eqref{gammauv} to $u$ and  $(1-\e)u_k$ inside $\Omega'$,
we get that
\[
\begin{split}
\e^n \int_{\Omega'} f &=\mu_{ \Gamma_{\e u}}(\Omega')\\
&\le \liminf_{k\to \infty} \mu_{\Gamma_{u-(1-\e)u_k}}(\Omega') \\
&\le \liminf_{k \to \infty} \int_{\Omega'\cap A_k^\e}\Big(f^{1/n}-(1-\e)f_k^{1/n}\Big)^n.
\end{split}
\] 
We now observe that, since $f_k$ converges to $f$ in $L^1_{\rm loc}(\Omega)$, we have
\[
\begin{split}
\left|\int_{\Omega'\cap A_k^\e}\Big(f^{1/n}-(1-\e)f_k^{1/n}\Big)^n -\int_{\Omega'\cap A_k^\e}\e^n f \right|&
\leq \int_{\Omega'}\left|\Big(f^{1/n}-(1-\e)f_k^{1/n}\Big)^n -\e^n f\right|  \to 0
\end{split}
\] 
as $k \to \infty$. Hence, combining the two estimates above, we immediately get
$$
\int_{\Omega'} f \leq \liminf_{k \to \infty}\int_{\Omega' \cap A_k^\e}  f,
$$
or equivalently
$$
\limsup_{k \to \infty} \int_{\Omega' \setminus A_k^\e}  f =0.
$$
Since $f \geq \lambda$ inside $\Omega$ (as a consequence of the fact that $f_k \geq \lambda$ inside $\Omega_k$),
this gives
\begin{equation}
\label{eq:meas to zero}
\lim_{k \to \infty} |\Omega' \setminus A_k^\e|=0 \qquad \forall\, \e \in (0,1).
\end{equation}
We now recall that, by the results in \cite{CA1,DepFi,DFS,S}, both $u$ and $(1-\e)u_k$ are strictly
convex and belong to $W^{2,1}(\Omega')$.
Hence we can apply \eqref{eq:order symm matrices2} to deduce that
\[
D^2 u - (1-\e)D^2u_k \geq 0 \qquad \text{a.e. on $A_k^\e$}.
\]
In particular, by \eqref{eq:meas to zero},
\[
|\Omega'\setminus\{ D^2u\ge (1-\e)D^2u_k\}|\to 0\qquad \text{as $k \to \infty$}.
\]
By a similar argument (exchanging the roles of $u$ and $u_k$)
\[
|\Omega'\setminus\{ D^2u_k \ge (1-\e)D^2u\}|\to 0\qquad \text{as $k \to \infty$}.
\]
Hence, if we call $B_k ^\e:=\left\{ x \in \Omega': \ (1-\e) D^2u_k \le D^2u \le \frac{1}{1-\e} D^2u_k\right\}$,
it holds
$$
\lim_{k \to \infty} |\Omega' \setminus B_k^\e|=0 \qquad \forall\, \e \in (0,1).
$$
Moveover, by \eqref{ulogu} applied to both $u_k$ and $u$, we have\footnote{If instead of
\eqref{ulogu} we only had uniform $L\log L$ a-priori estimates,
in place of H\"older inequality we could apply the elementary inequality $t\le\delta t \log (2+ t)+e^{1/\delta}$
with $t=|D^2 u-D^2 u_k|$ inside $\Omega'\setminus B^\e_k$,
and we would let first $k \to \infty$ and then send $\delta,\e\to 0$.}
\[
\begin{split}
\int_{\Omega'} |D^2 u-D^2 u_k|
&= \int_{\Omega'\cap B^\e_ k} |D^2 u-D^2 u_k|+\int_{\Omega'\setminus B^\e_k} |D^2 u-D^2 u_k|\\
&\le \frac{\e}{1-\e} \int_{\Omega'} |D^2 u|+ 
\|D^2 u-D^2u _k\|_{L^\gamma(\Omega')}|\Omega'\setminus B^\e_k|^{1-1/\gamma}\\
&\leq C\left(\frac{\e}{1-\e}+|\Omega'\setminus B^\e_k|^{1-1/\gamma}\right).
\end{split}
\]
Hence, letting first $k \to \infty$ and then sending $\e\to 0$, we obtain the desired result.
\end{proof}

\section{Proof of Theorem \ref{transport}}\label{sect:cor}

In order to prove Theorem \ref{transport}, we will need the following lemma (note that for the next result
we do not need to assume the convexity of the target domain):
\begin{lemma}\label{lemma:conv rhs}
Let $\O_1,\O_2\subset \R^n$ be two bounded domains, and let $f_k,g_k$ be a family of probability densities
such that $0<\l \leq f_k,g_k \le \Lambda$ inside $\O_1$ and $\O_2$ respectively.
Assume that $f_k \to f$ in $L^1(\Omega_1)$ and $g_k \to g$ in $L^1(\Omega_2)$,
and let $T_k:\O_1\to \O_2$ (resp. $T:\O_1\to \O_2$) be the (unique) optimal transport map for the
quadratic cost sending $f_k$ onto $g_k$ (resp. $f$ onto $g$).
Then 
$$
\frac{f_k}{g_k \circ T_k} \to \frac{f}{g\circ T} \qquad \text{in $L^1(\Omega_1)$}.
$$
\end{lemma}
\begin{proof}
By stability of optimal transport maps (see for instance \cite[Corollary 5.23]{V})
and the fact that $f_k\geq \l$ (and so $f \geq \l$), we know that $T_k \to T$
in measure (with respect to Lebesgue) inside $\Omega$.

We claim that $g \circ T_k \to g \circ T$ in  $L^1(\Omega_1)$. Indeed
this is obvious if $g$ is uniformly continous
(by the convergence in measure of $T_k$ to $T$). In the general case we choose $g_\eta \in C(\overline{\Omega}_2)$ such that
$\|g-g_\eta\|_{L^1(\Omega_2)} \leq \eta$ and we observe that (recall that $f_k,f \geq \l$, $g_k,g\leq \L$,
and that by definition of transport maps we have $T_\#f_k=g_k$, $T_\#f=g$)
\begin{align*}
\int_{\Omega_1} |g \circ T_k - g \circ T|
&\leq \int_{\Omega_1} |g_\eta \circ T_k - g_\eta \circ T| + \int_{\Omega_1} |g_\eta \circ T_k - g \circ T_k| \frac{f_k}{\lambda}
+ \int_{\Omega_1} |g_\eta \circ T - g \circ T| \frac{f}{\lambda}\\
&=  \int_{\Omega_1} |g_\eta \circ T_k - g_\eta \circ T|+ \int_{\Omega_2} |g_\eta  - g | \frac{g_k}{\lambda}
+ \int_{\Omega_2} |g_\eta  - g | \frac{g}{\lambda}\\
& \leq  \int_{\Omega_1} |g_\eta \circ T_k - g_\eta \circ T|+2 \frac{\L}{\l}\,\eta.
\end{align*}
Thus
$$
\limsup_{k \to \infty}\int_{\Omega_1} |g \circ T_k - g \circ T| \leq 2 \frac{\L}{\l}\,\eta,
$$
and the claim follows by the arbitrariness of $\eta$.

Since
\begin{align*}
\int_{\Omega_1}|g_k\circ T_k - g \circ T|&
\leq \int_{\Omega_1}|g_k\circ T_k - g \circ T_k| \frac{f_k}{\l}+\int_{\Omega_1}|g\circ T_k - g \circ T|\\
&= \int_{\Omega_2}|g_k - g| \frac{g_k}{\l}+\int_{\Omega_1}|g\circ T_k - g \circ T|\\
&\leq \frac{\L}{\l}\|g_k-g\|_{L^1(\Omega_2)}+\int_{\Omega_1}|g\circ T_k - g \circ T|,
\end{align*}
from the claim above we immediately deduce that also $g_k\circ T_k \to g \circ T$ in $L^1(\Omega_1)$.

Finally, since $g_k,g \geq \l$ and $f \leq \L$,
\begin{align*}
\int_{\Omega_1}\left|\frac{f_k}{g_k\circ T_k} -\frac{f}{g \circ T}\right|&
\leq \int_{\Omega_1}\left|\frac{f_k-f}{g_k\circ T_k}\right|
+ \int_{\Omega_1}f\left|\frac{1}{g_k\circ T_k}-\frac{1}{g \circ T}\right|\\
&\leq \frac{1}{\l} \|f_k-f\|_{L^1(\Omega_1)}+ \L \int_{\Omega_1} \frac{|g_k\circ T_k - g \circ T|}{g_k\circ T_k \,g\circ T}\\
&\leq \frac{1}{\l} \|f_k-f\|_{L^1(\Omega_1)}+ \frac{\L}{\l^2} \|g_k\circ T_k - g \circ T\|_{L^1(\Omega_1)},
\end{align*}
from which the desired result follows.
\end{proof}

\begin{proof}[Proof of Theorem \ref{transport}]
Since $T_k$ are uniformly bounded in $W^{1,\gamma}(\Omega_1')$ for any $\Omega_1'\subset \subset \Omega$,
it suffices to prove that $T_k \to T$ in $W^{1,1}_{\rm loc}(\Omega_1)$.

Fix $x_0 \in \Omega_1$ and $r>0$ such that $B_r(x_0)\subset \Omega_1$. By compactness, it suffices to show that there is an open neighborhood
$\mathcal U_{x_0}$ of $x_0$ such that
$\mathcal U_{x_0} \subset B_r(x_0)$ and
$$
\int_{\mathcal U_{x_0}} |T_k - T|+ |\nabla T_k - \nabla T| \to 0.
$$
It is well-known \cite{CA3} that $T_k$ (resp. $T$) can be written as $\nabla u_k$ (resp. $\nabla u$) for some
strictly convex function
$u_k:B_r(x_0) \to \R$ (resp. $u:B_r(x_0) \to \R$). Moreveor, up to subtract a constant
to $u_k$ (which
will not change the transport map $T_k$),
one may assume that $u_k(x_0)=u(x_0)$ for all $k \in \N$.

Since the functions $T_k=\nabla u_k$ are bounded (as they take values in the bounded set $\Omega_2$),
by classical stability of optimal maps (see for instance \cite[Corollary 5.23]{V})
we get that $\nabla u_k \to \nabla u$ in $L^1_{loc}(B_{r}(x_0))$.
(Actually, if one uses \cite{CA3}, $\nabla u_k$ are locally uniformly H\"older maps, so they converge locally uniformly
to $\nabla u$.)
Hence, to conclude the proof we only need to prove the convergence of $D^2 u_k$ to $D^2 u$ in a neighborhood of $x_0$.

To this aim, we observe that, by strict convexity of $u$, we can find a linear function
$\ell(z)=a \cdot z +b$ such that the open convex set $Z:=\{ z \,:\, u(z)< u(x_0)+ \ell(z)\}$ is non-empty and compactly supported
inside $B_{r/2}(x_0)$. Hence, by the uniform convergence of $u_k$ to $u$
(which follows from the $L^1_{\rm loc}$ convergence of the gradients,
the convexity of $u_k$ and $u$, and the fact that $u_k(x_0)=u(x_0)$), and the fact that $\nabla u$
is transversal to $\ell$ on $\partial Z$, we get that $Z_k:=\{ z \,:\, u_k(z)< u_k(x_0)+ \ell(z)\}$
are non-empty convex sets which converge in the Hausdorff distance to $Z$.

Moreover, 
by \cite{CA3} the maps $v_k:=u_k - \ell$ solve in the Alexandrov sense
$$
\begin{cases}
\det D^2 v_k=\frac{f_k}{g_k \circ T_k} \quad &\text{in $Z_k$}\\
v_k=0 &\text{on $\partial Z_k$}
\end{cases}
$$
(here we used that the Monge-Amp\`ere measures associated to $v_k$ and $u_k$ are the same).
Therefore, thanks to Lemma \ref{lemma:conv rhs}, we can apply Theorem \ref{conv}
to deduce that $D^2 u_k \to D^2 u$ in any relatively compact subset of $Z$, which concludes the proof.
\end{proof}

\end{document}